\documentclass[12pt]{amsart}
\usepackage{a4wide}
\usepackage[utf8]{inputenc}
\usepackage{amsmath}
\usepackage{amssymb}
\usepackage{amsfonts}
\usepackage{amsopn}
\usepackage{graphicx}
\usepackage{enumerate}
\usepackage{color}
\usepackage{mathtools}
\usepackage{mathrsfs}
\usepackage{bbm}
\usepackage{yhmath}
\usepackage{cite}
\usepackage[colorlinks,linkcolor=blue,anchorcolor=blue,citecolor=blue]{hyperref}

\newtheorem{theorem}{Theorem}[section]
\newtheorem{proposition}[theorem]{Proposition}

\newtheorem{definition}[theorem]{Definition}
\newtheorem{corollary}[theorem]{Corollary}
\theoremstyle{definition}
\newtheorem{example}[theorem]{Example}

\numberwithin{equation}{section}

\newcommand{\mcal}{\mathcal}
\newcommand{\mscr}{\mathscr}
\newcommand{\set}[1]{\left\{ #1 \right\}}

\newcommand{\R}{\mathbb{R}}

\newcommand{\Z}{\mathbb{Z}}
\newcommand{\N}{\mathbb{N}}
\newcommand{\PP}{\mathbb{P}}

\newcommand{\f}{\infty}

\newcommand{\wh}[1]{\widehat{#1}}
\newcommand{\wt}[1]{\widetilde{#1}}

\newcommand{\ep}{\varepsilon}
\newcommand{\sse}{\subseteq}
\newcommand{\sm}{\setminus}
\newcommand{\D}{\;\mathrm{d}}

\title[weak convergence and spectrality]{Weak Convergence and Spectrality of Infinite Convolutions}

\author[W. Li]{Wenxia Li}
\address[W. Li]{School of Mathematical Sciences, Shanghai Key Laboratory of PMMP, East China Normal University, Shanghai 200241, People's Republic of China}
\email{wxli@math.ecnu.edu.cn}

\author[J. J. Miao]{Jun Jie Miao}
\address[J. J. Miao]{School of Mathematical Sciences, Shanghai Key Laboratory of PMMP, East China Normal University, Shanghai 200241, People's Republic of China}
\email{jjmiao@math.ecnu.edu.cn}

\author[Z. Wang]{Zhiqiang Wang*}
\address[Z. Wang]{School of Mathematical Sciences, Shanghai Key Laboratory of PMMP, East China Normal University, Shanghai 200241,
People's Republic of China}
\email{zhiqiangwzy@163.com}

\thanks{* Corresponding author}

\begin{document}

\begin{abstract}
Let $\{ A_k\}_{k=1}^\infty$ be a sequence of finite subsets of $\mathbb{R}^d$ satisfying that $\# A_k \ge 2$ for all integers $k \ge 1$.
In this paper,  we first give a sufficient and necessary condition
for the existence of the infinite convolution
$$\nu =\delta_{A_1}*\delta_{A_2} * \cdots *\delta_{A_n}*\cdots, $$
where all sets $A_k \sse \R_+^d$ and $\delta_A = \frac{1}{\# A} \sum_{a \in A} \delta_a$.
Then we study the spectrality of a class of infinite convolutions generated by Hadamard triples in $\mathbb{R}$ and construct a class of singular spectral measures without compact support.
Finally we show that such measures are abundant, and the dimension of their supports has the intermediate-value property.
\end{abstract}


\maketitle

\section{Introduction}

\subsection{Spectral measures and frame spectral measures}
Let $\nu$ be a Borel probability measure on $\R^d$. We call $\nu$ a \emph{spectral measure} if there exists a countable subset $\Lambda \sse \R^d$ such that the family of exponential functions $$\set{e_\lambda(x) = e^{-2\pi i \lambda \cdot x}: \lambda \in \Lambda}$$ forms an orthonormal basis in $L^2(\nu)$.
The set $\Lambda$ is called a \emph{spectrum} of $\nu$.
The existence of spectrum of $\nu$ is a basic question in harmonic analysis, and it was initiated in the paper of Fuglede~\cite{Fuglede-1974}, where he considered the normalized Lebesgue measure on measurable subsets.
Fuglede conjectured that
\begin{quote}
  \emph{A measurable set $\Gamma \sse \R^d$ with positive finite Lebesgue measure is a spectral set, that is, the normalized Lebesgue measure on $\Gamma$ is a spectral measure, if and only if $\Gamma$ tiles $\R^d$ by translations.}
\end{quote}
Although the both directions of Fuglede's conjecture were disproved by Tao \cite{Tao-2004} and the others \cite{Farkas-Matolcsi-Mora-2006,Farkas-Revesz-2006,Kolountzakis-Matolcsi-2006a,Kolountzakis-Matolcsi-2006b,Matolcsi-2005}
for $d \ge 3$,
the connection between spectrality and tiling has generated a lot of interest, and some affirmative results have been proved for special cases \cite{Iosevich-Katz-Tao-2003,Laba-2001}.
Fuglede's conjecture remains open for $d=1$ and $d=2$.
The study of spectral sets is also closely related to Gabor bases \cite{Liu-Wang-2003} and wavelets \cite{Wang-2002}.

A natural generalization of exponential orthonormal bases is the Fourier frame.
A Borel probability measure $\nu$ on $\R^d$ is called a \emph{frame spectral measure} if there exists $\Lambda \subseteq \R^d$ and $0< A \le B < \f$ such that
$$ A \| f\|^2 \le \sum_{\lambda \in \Lambda} \left| \int_{\R^d} f(x) e^{2\pi i \lambda \cdot x} \D \nu(x) \right|^2 \le B \| f\|^2$$ for all $f \in L^2(\nu)$.
The set $\Lambda$ is called a \emph{frame spectrum} of $\nu$, and the family of exponential functions $$\set{e_\lambda(x) = e^{-2\pi i \lambda \cdot x}: \lambda \in \Lambda}$$ is called a \emph{Fourier frame} of $L^2(\nu)$.
Fourier frames were introduced in 1952 by Duffin and Schaeffer \cite{Duffin-Schaeffer-1952} to study nonharmonic Fourier series.
Now, frames in Hilbert spaces have been a fundamental research area in applied harmonic analysis.
Because of overcompleteness, the reconstruction process of frames is more robust, which is very useful in signal processing.
Readers can refer to \cite{Christensen-2003} for the general frame theory.

In 1998, Jorgensen and Pedersen~\cite{Jorgensen-Pedersen-1998} discovered that some self-similar measures may also have spectra. A simple example is that the self-similar measure on $\R$ given by the identity $$\mu(\;\cdot\;) = \frac{1}{2} \mu(4\;\cdot\;) + \frac{1}{2} \mu(4\;\cdot\; -2)$$ is a spectral measure with a spectrum $$\Lambda = \bigcup_{n=1}^\f \set{\ell_1 + 4\ell_2 + \cdots + 4^{n-1} \ell_n: \ell_1,\ell_2,\cdots,\ell_n \in \set{0,1}}.$$
From then on, the spectrality and non-spectrality of various singular fractal measures, such as self-affine measures  and Cantor-Moran measures, have been extensively studied, see \cite{An-Fu-Lai-2019,An-He-He-2019,An-He-2014,An-He-Lau-2015,An-He-Li-2015,An-Wang-2021,Dai-He-Lau-2014,Deng-Chen-2021,
Dutkay-Haussermann-Lai-2019,Dutkay-Lai-2014,Dutkay-Lai-2017,Fu-Wen-2017,Laba-Wang-2002,Li-2009,Liu-Dong-Li-2017,
He-Tang-Wu-2019,LMW} and references therein for details.
In \cite{Lai-Wang-2017}, Lai and Wang demonstrated the first non-trivial example of singular fractal measure which is a frame spectral measure but not a spectral measure.
Readers may refer to~\cite{Falco03} for the details on fractal geometry.

There is a big difference between absolutely continuous measures and singular measures in the theory of spectrum and frame spectrum.
In the classical case, for the normalized Lebesgue measure on a finite interval, the frame spectrum is essentially characterized by the Beurling lower uniform density \cite{Landau-1967} and the complete description of frame spectra is given by Ortega-Cerd\'{a} and Seip \cite{Ortega-Seip-2002}.
Dutkay, Han, Sun, and Weber \cite{Dutkay-Han-Sun-Weber-2011} related the Beurling dimension of frame spectrum of self-similar measures to the Hausdorff dimension of self-similar sets.
Recently,  Shi \cite{Shi-2021} proved that for any Borel probability measure on $\R^d$, the Beurling dimension of frame spectrum is bounded by the upper entropy dimension of measure.
In particular, for self-affine measures, the Beurling dimension of frame spectrum (of course, of spectrum) is less than or equal to the Hausdorff dimension of measure.
Comparing with absolutely continuous measures, An and Lai \cite{An-Lai-2020} showed that all singular self-affine spectral measures generated by Hadamard triples have an arbitrarily sparse spectrum in the sense that the Beurling dimension is zero.

Associated with singular spectral measures, Strichartz \cite{Strichartz-2000,Strichartz-2006} studied the mock Fourier series and transforms, and showed the better convergence properties for some self-similar measures: the mock Fourier series of continuous functions converge uniformly, and the mock Fourier series of $L^p$-functions also converge in the $L^p$-norm.
A surprising phenomenon is that the convergence of the mock Fourier series may be very different for distinct spectra of a singular spectral measure \cite{Dutkay-Han-Sun-2014}.

Currently, spectral measures studied in the literature, to the best of our knowledge, are compactly supported.
In \cite{Nitzan-2016}, Nitzan et al. constructed the frame spectrum of the normalized Lebesgue measure on any unbounded subset with finite measure.
This motivates us to construct spectral measures without compact support.
For absolutely continuous measures, here is a trivial example.
Let $$\Gamma = \bigcup_{n=0}^\f \left(n+\frac{1}{n+2}, n+\frac{1}{n+1} \right].$$
It is easy to check that the set $\Gamma$ tiles $\R$ by the lattice $\Z$.
By the Fuglede's result in \cite{Fuglede-1974}, $\Gamma$ is a spectral set, that is, the Lebesgue measure on $\Gamma$ is a spectral measure.
Obviously, the set $\Gamma$ is unbounded.

In this paper we focus on constructing singular spectral measures which may not be compactly supported (see Theorem \ref{main-theorem}). Moreover, we show that such measures are abundant, and the dimension of their supports has the intermediate-value property (see Theorem \ref{dimension-theorem}).

\subsection{Infinite convolutions}
Let $\#$ denote the cardinality of a set, and let $\delta_a$ denote the Dirac measure at the point $a$.
For a finite subset $A \sse \R^d$, we define the uniform discrete measure supported on $A$ by
\begin{equation*}
\delta_A = \frac{1}{\# A} \sum_{a \in A} \delta_a.
\end{equation*}
Let $\{ A_k\}_{k=1}^\f$ be a sequence of finite subsets of $\R^d$ such that $\# A_k \ge 2$ for every $k \ge 1$.
For  each integer $n \geq 1$, we define
\begin{equation}\label{discrete-convolution}
  \nu_n =\delta_{A_1}*\delta_{A_2} * \cdots *\delta_{A_n},
\end{equation}
where the notation $*$ denotes the convolution between measures.
If the sequence of convolutions $\{\nu_n\}$ converges weakly to a Borel probability measure, then the weak limit is called the \emph{infinite convolution} of $\{\delta_{A_k}\}$, denoted by
\begin{equation}\label{infinite-convolution}
  \nu =\delta_{A_1}*\delta_{A_2} * \cdots *\delta_{A_k} *\cdots.
\end{equation}

A famous example of infinite convolutions is the so-called \textit{infinite Bernoulli convolution}
$$\nu_\lambda = \delta_{\{\pm \lambda\}} * \delta_{\{\pm \lambda^2\}} * \cdots *\delta_{\{\pm \lambda^k\}}*\cdots$$
with $\lambda\in(0,1)$, that is, $A_k=\{-\lambda^k,\lambda^k\}$ in \eqref{infinite-convolution}.
The measure $\nu_\lambda$ can be also seen as the distribution of $\sum _{k=1}^\infty \pm \lambda^k$ where the signs are chosen independently with probability $1/2$. Infinite Bernoulli convolutions have been studied since 1930's around a fundamental question which is to determine the absolute continuity and singularity of $\nu_\lambda$ for $\lambda \in (1/2, 1)$. Furthermore, if $\nu_\lambda$ is absolutely continuous, the smoothness of the density is explored; if $\nu_\lambda$ is singular, the dimension of measure is investigated.
These studies revealed connections with  harmonic analysis, the theory of algebraic numbers, dynamical systems, and fractal geometry \cite{Chan-Ngai-Teplyaev-2015,Hu-Lau-2019,Shmerkin-2014,Solomyak1995}. Readers can refer to the  excellent survey paper \cite{Peres-Schlag-Solomyak-2000} for this topic.
For the spectrality of infinite Bernoulli convolutions, Dai \cite{Dai-2012} showed that $\nu_\lambda$ is a spectral measure if and only if $\lambda = 1/(2q)$ for some positive integer $q$, and then Dai, He, and Lau \cite{Dai-He-Lau-2014} generalized it to the case of $N$-Bernoulli convolution.

The infinite convolution based on Hadamard triples was first raised by Strichartz~\cite{Strichartz-2000} to study the spectrality.
From then on, the following spectral problem of infinite convolutions is widely studied.
\begin{quote}
\emph{Given a sequence of Hadamard triples $\{(R_k, B_k, L_k)\}_{k=1}^\f$, under what assumptions is the infinite convolution
$$ \mu = \delta_{R_1^{-1} B_1} * \delta_{(R_2R_1)^{-1} B_2} * \cdots * \delta_{(R_n \cdots R_2 R_1)^{-1} B_n} * \cdots $$
spectral?}
\end{quote}
Many affirmative partial results have been obtained for this question, see~\cite{An-Fu-Lai-2019,An-He-He-2019,An-He-2014,An-He-Lau-2015,An-He-Li-2015,Dutkay-Haussermann-Lai-2019,
Dutkay-Lai-2017,Fu-Wen-2017,Laba-Wang-2002}.
If all Hadamard triples $(R_k,B_k,L_k)=(R,B,L)$, then the infinite convolution reduces to a self-affine measure.
Dutkay, Haussermann and Lai \cite{Dutkay-Haussermann-Lai-2019} proved that the self-affine measure with equal weights generated by a Hadamard triple in $\R^d$ is a spectral measure, and this result for the one-dimensional case was proved by {\L}aba and Wang \cite{Laba-Wang-2002}.

\subsection{Main results}
Let $\nu _n$ and $\nu $ be defined by (\ref{discrete-convolution}) and \eqref{infinite-convolution}, respectively.
To study the spectrality of $\nu$, the following question inevitably needs to be answered first.
\begin{center}
{\it With what conditions does the infinite convolution $\nu$ exist as the weak limit of $\{\nu_n\}$?}
\end{center}
In the first part of this paper, we obtain some easy-checked criteria (see Theorems \ref{weak-convergence-infinite-convolution}, \ref{weak-convergence-thm1} and Corollary \ref{cor_1}).

For $a=(a_1, a_2, \cdots, a_d) \in \R^d$, we denote $|a| = \sqrt{a_1^2 + a_2^2 + \cdots + a_d^2}$.
It is well-known that if
\begin{equation}\label{weak-convergence-condition}
  \sum_{k=1}^{\f} \max\set{|a|: a\in A_k} < \f,
\end{equation}
then the infinite convolution $\nu$ exists, and moreover, the support of  $\nu$ is compact.
In~\cite{An-He-He-2019}, the authors conjectured that in $\R$, if $A_k \sse [0,+\f)$ and $\# A_k =3$ for every $k \ge 1$, the assumption \eqref{weak-convergence-condition} is also necessary for the existence of infinite convolutions.
The following Corollary \ref{cor_1} gives an affirmative conclusion.

Write $\R_+ = [0,+\f)$, and we first provide a sufficient and necessary condition for the existence of infinite convolutions when all sets $A_k \sse \R_+^d$.

\begin{theorem}\label{weak-convergence-infinite-convolution}
  Let $\{ A_k\}_{k=1}^\f$ be a sequence of finite subsets of $\R_+^d$ satisfying that $\# A_k \ge 2$ for each $k \ge 1$.
  Let $\nu_n$ be defined in \eqref{discrete-convolution}.
  Then the sequence of convolutions $\set{\nu_n}$ converges weakly to a Borel probability measure if and only if
  \begin{equation}\label{eq-thm-1-1}
    \sum_{k=1}^{\f} \frac{1}{\# A_k} \sum_{a\in A_k} \frac{|a|}{1+|a|} < \f.
  \end{equation}
\end{theorem}

Moreover, if the cardinalities of all sets $A_k$ are bounded, as a consequence of Theorem~\ref{weak-convergence-infinite-convolution}, we show that \eqref{weak-convergence-condition} is also a necessary condition for the weak convergence of $\{\nu_n\}$.

\begin{corollary}\label{cor_1}
  Let $\{ A_k\}_{k=1}^\f$ be a sequence of finite subsets of $\R_+^d$ satisfying that $\# A_k \ge 2$ for each $k \ge 1$.
  Let $\nu_n$ be defined in \eqref{discrete-convolution}.
  Suppose that $\sup\set{\# A_k: k \ge 1} < \f$.
  Then the sequence of convolutions $\{\nu_n\}$ converges weakly to a Borel probability measure if and only if
  $$\sum_{k=1}^{\f} \max\set{ |a| : a\in A_k} < \f. $$
\end{corollary}

For the general case that all sets $A_k \sse \R^d$, the following theorem provides a sufficient condition, which may be easier to verify.

\begin{theorem}\label{weak-convergence-thm1}
  Let $\{ A_k\}_{k=1}^\f$ be a sequence of finite subsets of $\R^d$ satisfying that $\# A_k \ge 2$ for each $k \ge 1$.
  Let $\nu_n$ be defined in \eqref{discrete-convolution}.
  Suppose that
  \begin{equation}\label{eq-thm-1-3-1}
    \sum_{k=1}^{\f} \max\set{|a|^2: a\in A_k} < \f,
  \end{equation}
  and the series
  \begin{equation}\label{eq-thm-1-3-2}
    \sum_{k=1}^{\f} \frac{1}{\# A_k} \sum_{ a \in A_k} a
  \end{equation}
  converges in $\R^d$.
  Then the infinite convolution $\nu$ defined in \eqref{infinite-convolution} exists.
\end{theorem}

The spectrality of infinite convolutions in $\R^d$ is very complicated.
Next, we focus on the spectrality of a class of infinite convolutions generated by Hadamard triples in $\R$ which may not be compactly supported.
We point out that the product of spectral measures in $\R$ gives a spectral measure in higher dimension.

Let $\set{b_k}$ and $\set{N_k}$ be two sequences of positive integers satisfying
\begin{equation}\label{b-k-N-k}
  N_k \ge b_k\ge 2 \;\text{ and } \; b_k \mid N_k.
\end{equation}
Let $\set{B_k}$ be a sequence of finite subsets of nonnegative integers.
We say that $\set{B_k}$ is a sequence of \emph{nearly consecutive digit sets} with respect to $\set{b_k}$ and $\set{N_k}$ if
\begin{equation*}
  B_k \equiv \set{0,1,2,\cdots, b_k-1} \pmod{N_k}
\end{equation*}
for every $k \ge 1$, and
\begin{equation*}
  \sum_{k=1}^\f \frac{1}{b_k} \#\big(B_k \sm \set{0,1,2,\cdots, b_k-1} \big) < \f.
\end{equation*}
We prove that infinite convolutions generated by a sequence of nearly consecutive digit sets are spectral measures.

\begin{theorem}\label{main-theorem}
  Given two sequences $\set{b_k}$ and $\set{N_k}$ of positive integers satisfying \eqref{b-k-N-k}, suppose that $\set{B_k}$ is a sequence of nearly consecutive digit sets with respect to $\set{b_k}$ and $\set{N_k}$.
  Then the infinite convolution $$\mu = \delta_{N_1^{-1} B_1} * \delta_{(N_1N_2)^{-1} B_2} * \cdots * \delta_{(N_1N_2\cdots N_n)^{-1} B_n} * \cdots $$ exists, and moreover, $\mu$ is a spectral measure.
\end{theorem}

The following corollary is an immediate consequence of Theorem \ref{main-theorem}, which has been proved by An and He in~\cite{An-He-2014}.
\begin{corollary}
  Let $\set{b_k}$ and $\set{N_k}$ be two sequences of positive integers satisfying \eqref{b-k-N-k}, and $B_k = \set{0,1,2,\cdots, b_k-1}$ for each $k \ge 1$.
  Then the infinite convolution $$\mu= \delta_{N_1^{-1} B_1} * \delta_{(N_1N_2)^{-1} B_2} * \cdots * \delta_{(N_1N_2\cdots N_n)^{-1} B_n}* \cdots$$ is a spectral measure.
\end{corollary}

Let $\mu $ be a Borel probability measure on $\R^d$, the support of $\mu$, denoted by $\mathrm{spt}(\mu)$, is the smallest closed subset with full measure, i.e.,
$$\mathrm{spt}(\mu) = \R^d \sm \bigcup \set{ U\sse \R^d: \text{ $U$ is open, and $\mu(U)=0$}}.$$
We apply Theorem \ref{main-theorem} to construct an example of spectral measure without compact support in $\R$.

\begin{example}\label{expl}
  Let $b_k = (k+1)^2$ and $N_k = 2 b_k$ for each $k \ge 1$.
  We write $$ B_k = \set{ 0,1,\cdots, b_k-2, b_k -1 + N_1 N_2 \cdots N_k } $$ for every $k \ge 1$.
Clearly, $\set{B_k}$ is a sequence of nearly consecutive digit sets with respect to $\set{b_k}$ and $\set{N_k}$.
Hence, by Theorem \ref{main-theorem}, the infinite convolution $$\mu= \delta_{N_1^{-1} B_1} * \delta_{(N_1N_2)^{-1} B_2} * \cdots * \delta_{(N_1N_2\cdots N_n)^{-1} B_n}* \cdots$$
is a spectral measure.
Since $$ \sum_{k=1}^{\f} \frac{\max\set{b: b \in B_k}}{N_1 N_2 \cdots N_k} =\sum_{k=1}^{\f} \left( \frac{b_k -1}{N_1 N_2 \cdots N_k} + 1 \right) =\f, $$
by simple calculation, the support of the measure $\mu$ is \emph{not} compact.
\end{example}

Actually, there are abundant examples of singular spectral measures without compact support, and the dimension of their supports has the intermediate-value property. The similar result for spectral measures with compact support has been proved in \cite{Dai-Sun-2015}.
The intermediate-value property also holds for the entropy of $K$-partitions in any Bernoulli system \cite{Lind-Peres-Schlag-2002}.

\begin{theorem}\label{dimension-theorem}
  For $0\le \alpha \le \beta \le 1$, there exists a singular spectral measure $\mu$ without compact support such that
    $$
    \dim_H \mathrm{spt}(\mu) = \alpha\leq \beta =\dim_P \mathrm{spt}(\mu),
    $$
where $\dim_H$ and $\dim_P$ denote the Hausdorff dimension and packing dimension, respectively.
\end{theorem}

The rest of the paper is arranged as follows. In Section \ref{pre}, we recall some definitions and some known results. In Section \ref{weak-convergence}, we study the weak convergence of infinite convolutions, and give the proofs of Theorem \ref{weak-convergence-infinite-convolution}, Corollary \ref{cor_1} and Theorem \ref{weak-convergence-thm1}. We study the spectrality of a class of infinite convolutions and prove Theorem \ref{main-theorem} in Section \ref{spectrality}. In the last Section \ref{intermediate-value}, we construct singular spectral measures without compact support, and give the proof of Theorem \ref{dimension-theorem}.

\section{Preliminaries}\label{pre}

We use $\mcal{P}(\R^d)$ to denote the set of all Borel probability measures on $\R^d$.
For $\mu \in \mcal{P}(\R^d)$, the \emph{Fourier transform} of $\mu$ is given by
$$ \wh{\mu}(\xi) = \int_{\R^d} e^{-2\pi i \xi \cdot x} \D \mu(x). $$

Let $\mu,\mu_1,\mu_2,\cdots \in \mcal{P}(\R^d)$. We say that $\mu_n$ \textit{converges weakly} to $\mu$ if $$\lim_{n \to \f} \int_{\R^d} f(x) \D \mu_n(x) = \int_{\R^d} f(x) \D \mu(x),$$
for all $ f \in C_b(\R^d),$ where $C_b(\R^d)$ is the set of all bounded continuous functions on $\R^d$.
The following well-known theorem characterizes the weak convergence of probability measures.
We refer readers to \cite{Bil03} for details.

\begin{theorem}\label{weak-convergence-theorem}
  Let $\mu,\mu_1,\mu_2,\cdots \in \mcal{P}(\R^d)$. The following statements are equivalent.

  {\rm(i)} $\mu_n$ \textit{converges weakly} to $\mu$;

  {\rm(ii)} For every closed subset $F \sse \R^d$, $\displaystyle \mu(F) \ge \limsup_{n \to \f} \mu_n(F)$;

  {\rm(iii)} For every open subset $U \sse \R^d$, $\displaystyle \mu(U) \le \liminf_{n \to \f} \mu_n(U)$.
\end{theorem}

For $\mu,\nu \in \mcal{P}(\R^d)$, the \emph{convolution} $\mu *\nu$ is the unique Borel probability measure satisfying $$\int_{\R^d} f(x) \D \mu*\nu(x) = \int_{\R^d \times \R^d} f(x+y) \D \mu \times \nu(x,y),$$
for  all $ f \in C_b(\R^d)$.
It is straightforward that $$\wh{\mu *\nu}(\xi) = \wh{\mu}(\xi) \wh{\nu}(\xi).$$

Next we introduce the Hadamard triple, which is a primitive assumption to construct singular spectral measures.

\begin{definition}\label{Hadamard-triple}
  Let $R \in M_d(\Z)$ be a $d \times d$ expansive matrix (i.e., all eigenvalues have modulus strictly greater than $1$) with integer entries.
  Let $B,L \sse \Z^d$ be finite subsets of integer vectors with $\# B = \# L \ge 2$.
  If the matrix $$ \left[ \frac{1}{\sqrt{\# B}} e^{-2 \pi i  (R^{-1}b)\cdot\ell }  \right]_{b \in B, \ell \in L} $$ is unitary,
  we call $(R, B, L)$ a {\it Hadamard triple} in $\R^d$.
\end{definition}

For example, for $d=1$, given $b,N \in \Z$ with $b\ge 2$ and $b \mid N$, let $$B=\set{0,1,\cdots, b-1},\quad L =\set{0,\frac{N}{b}, \frac{2N}{b},\cdots, \frac{(b-1)N}{b} }. $$
Then $(N,B,L)$ is a Hadamard triple in $\R$.

In fact, $(R, B, L)$ is a Hadamard triple in $\R^d$ if and only if the set $L$ is a spectrum of the discrete measure $\delta_{R^{-1} B}$.
See~\cite{Dutkay-Haussermann-Lai-2019} for more details about Hadamard triples.

\section{Weak convergence of convolutions}\label{weak-convergence}

Let $X$ be a random vector on a probability space $(\Omega,\mscr{F},\PP)$. The image measure, denoted by $\mu$, of $\PP$ under $X$ is called the \emph{distribution} of $X$, i.e., for every Borel subset $E \sse \R^d$, $$\mu(E) = \PP(X^{-1}E).$$

For $\mu \in \mathcal{P}(\R^d)$, we define $$ c(\mu) = \int_{\R^d} x \D \mu(x), \quad \mathrm{M}(\mu) = \int_{\R^d} |x - c(\mu)|^2 \D \mu(x). $$
It is easy to check that $$ \mathrm{M}(\mu) = \int_{\R^d} |x|^2 \D \mu(x) - |c(\mu)|^2. $$

For $\mu \in \mathcal{P}(\R^d)$ and $r>0$, we define a new Borel probability measure $\mu_r$ by
\begin{equation}\label{mu_r}
  \mu_r(E) = \mu\big( E \cap B(r) \big) + \mu\big( \R^d \setminus B(r) \big)\delta_0(E)
\end{equation}
for every Borel subset $E \sse \R^d$, where $\delta_0$ denotes the Dirac measure at $0$ and $B(r)$ denotes the closed ball with center at $0$ and radius $r$.

Let $\mu_1, \mu_2, \cdots$ be a sequence of Borel probability measures on $\R^d$.
By the existence theorem of product measures, there exists a probability space $(\Omega, \mathscr{F},\PP)$ and a sequence of independent random vectors $\{ X_k \}_{k=1}^\f$ such that for each $k \ge 1$ the distribution of $X_k$ is $\mu_k$.
Then one may check that $\mu_1 * \mu_2 * \cdots *\mu_n$ is the distribution of $X_1 + X_2 + \cdots + X_n$.
Thus, the existence of infinite convolution $\mu_1 * \mu_2* \cdots $ is equivalent to the convergence of the series $X_1 + X_2 + \cdots $ in distribution.
It is well-known that the convergence in distribution and the almost sure convergence of the sum of independent random vectors are equivalent.
The following theorem is a consequence of Kolmogorov’s three series theorem, or see Theorem 34 in \cite{Jessen-Wintner-1935} for the proof.

\begin{theorem}\label{three-series-theorem}
  Let $\mu_1, \mu_2, \cdots$ be a sequence of Borel probability measures on $\R^d$.
  Fix a constant $r>0$, and let $\mu_{k,r}$ be defined by \eqref{mu_r} for the measure $\mu_k$, $k \ge 1$.
  Then the sequence of convolutions $\set{\mu_1 * \mu_2 * \cdots *\mu_n}$ converges weakly to a Borel probability measure, i.e., the infinite convolution $\mu_1 *\mu_2 *\cdots $ exists, if and only if the following three series all converge:
  $$ (i) \sum_{k=1}^{\f} \mu_k\big( \R^d \setminus B(r)\big),\quad (ii) \sum_{k=1}^{\f} c(\mu_{k,r}), \quad (iii) \sum_{k=1}^{\f} \mathrm{M}(\mu_{k,r}). $$
\end{theorem}

\begin{proof}[Proof of Theorem \ref{weak-convergence-infinite-convolution}]
Write $\mu_k = \delta_{A_k}$ for each $k \ge 1$.
Fix $r=1$, and let $\mu_{k,r}$ be defined by \eqref{mu_r} for the measure $\mu_k$, $k \ge 1$.
By Theorem \ref{three-series-theorem}, the weak convergence of $\set{\nu_n}$ is equivalent to the convergence of the following three series
  \begin{equation}\label{three-series}
     \sum_{k=1}^{\f} \mu_k\big( \R^d \setminus B(r) \big),\quad \sum_{k=1}^{\f} c(\mu_{k,r}), \quad \sum_{k=1}^{\f} \mathrm{M}(\mu_{k,r}).
  \end{equation}
For each integer $k \ge 1$, we write $A_{k,1} = \{ a \in A_k: |a| \le 1\}$ and $A_{k,2} = \{ a \in A_k: |a| > 1\}$. It is clear that  $A_k=A_{k,1}\cup A_{k,2}$.

We first prove the sufficiency.  Assume that \eqref{eq-thm-1-1} holds.
By the fact  that $\frac{2t}{1+t} >1$ for $t >1$, we have that
\begin{eqnarray*}
\sum_{k=1}^{\f} \mu_k \big( \R^d \setminus B(r) \big) &=& \sum_{k=1}^{\f} \frac{\#A_{k,2}}{\# A_k} \\
&\le& \sum_{k=1}^{\f} \frac{1}{\# A_k} \sum_{a\in A_{k,2}} \frac{2|a|}{1+|a|} \\
&\le& 2\sum_{k=1}^{\f} \frac{1}{\# A_k} \sum_{a\in A_k} \frac{|a|}{1+|a|} \\
&<& \f.
\end{eqnarray*}
Using the fact that $\frac{2}{1+t} \ge 1$ for $0\le t \le 1$, we have that
\begin{eqnarray*}
\sum_{k=1}^{\f} \int_{\R^d} |x| \D \mu_{k,r}(x) & = & \sum_{k=1}^{\f} \frac{1}{\#A_k} \sum_{a \in A_{k,1}} |a| \\
&\le& \sum_{k=1}^{\f} \frac{1}{\#A_k} \sum_{a \in A_{k,1}} \frac{2}{1+|a|} |a|\\
&\le& 2\sum_{k=1}^{\f} \frac{1}{\# A_k} \sum_{a\in A_k} \frac{|a|}{1+|a|} \\
&<& \f.
\end{eqnarray*}
It follows that $$\sum_{k=1}^{\f} |c(\mu_{k,r})| \le \sum_{k=1}^{\f} \int_{\R^d} |x| \D \mu_{k,r}(x) < \f, $$
and $$\sum_{k=1}^{\f} \mathrm{M}(\mu_{k,r}) \le \sum_{k=1}^{\f} \int_{\R^d} |x|^2 \D \mu_{k,r}(x) \le \sum_{k=1}^{\f} \int_{\R^d} |x| \D \mu_{k,r}(x) < \f.$$
Hence, all three series in \eqref{three-series} converge.

Next we prove the necessity. Assume that the three series in \eqref{three-series} are convergent.
Note that $$c(\mu_{k,r}) = \frac{1}{\# A_k} \sum_{a \in A_{k,1}} a. $$
Since all sets $A_k \sse \R_+^d$ , it follows from the convergence of $\sum_{k=1}^{\f} c(\mu_{k,r})$ that
$$\sum_{k=1}^{\f} \frac{1}{\# A_k} \sum_{a \in A_{k,1}} \frac{|a|}{1+|a|} \le \sum_{k=1}^{\f}\frac{1}{\# A_k} \sum_{a \in A_{k,1}} |a| < \f. $$
On the other hand,
$$\sum_{k=1}^{\f} \frac{1}{\# A_k} \sum_{a \in A_{k,2}} \frac{|a|}{1+|a|} \le \sum_{k=1}^{\f}\frac{\# A_{k,2}}{\# A_k} = \sum_{k=1}^{\f} \mu_k\big( \R^d \setminus B(r) \big) < \f. $$
Noting that $A_k=A_{k,1}\cup A_{k,2}$, we immediately have that $$ \sum_{k=1}^{\f} \frac{1}{\# A_k} \sum_{a\in A_k} \frac{|a|}{1+|a|} < \f,$$
that is, \eqref{eq-thm-1-1} holds.

Therefore, the conclusion holds
\end{proof}

\begin{proof}[Proof of Corollary \ref{cor_1}]
By Theorem~\ref{weak-convergence-infinite-convolution}, the sufficiency is straightforward.

For the necessity, we assume that the sequence of convolutions $\{\nu_n\}$ converges weakly to a Borel probability measure.
Let $N= \sup\set{\# A_k: k \ge 1}$.
For each integer $k \ge 1$, we write $m_k = \max\{|a|: a \in A_k\}$ .
By Theorem \ref{weak-convergence-infinite-convolution}, we have that
$$ \frac{1}{N}\sum_{k=1}^{\f} \frac{m_k}{1+m_k} \le \sum_{k=1}^{\f} \frac{1}{\# A_k} \sum_{a\in A_k} \frac{|a|}{1+|a|} < \f.$$
This implies $$\lim_{k \to \f} \frac{m_k}{1+m_k} = 0.$$
Clearly, the sequence  $\{m_k\} $ converges to $0$ as $k$ tends to $\f$. So we may find $k_0 \ge 1$ such that $m_k < 1/2$ for all $k \ge k_0$.
  Thus, we have that $$\sum_{k=k_0}^{\f} m_k \le \frac{3}{2}\sum_{k=k_0}^{\f} \frac{m_k}{1+m_k} < \f,$$
  and it implies that
  $$\sum_{k=1}^{\f} \max\set{|a|: a\in A_k} < \f.$$
Therefore, the necessity holds.
\end{proof}

\begin{proof}[Proof of Theorem~\ref{weak-convergence-thm1}]
  Write $\mu_k = \delta_{A_k}$ for each $k \ge 1$. Fix $r=1$, and let $\mu_{k,r}$ be defined by \eqref{mu_r} for the measure $\mu_k$, $k \ge 1$.

  Since the series in \eqref{eq-thm-1-3-1} converges, we have that $\max\set{|a|: a\in A_k}$ converges to $0$ as $k$ tends to $\f$.
  We choose an integer $k_0\ge 1$ such that $\max\set{|a|: a\in A_k} < 1$ for $k \ge k_0$.
  It follows that $\mu_{k,r} = \mu_k$ for $k \ge k_0$.
  Thus, we have that $$ \sum_{k=k_0}^{\f} \mu_k\big( \R^d \setminus B(r) \big) =0,$$
  and $$ \sum_{k=k_0}^{\f} \mathrm{M}(\mu_{k,r}) =\sum_{k=k_0}^{\f} \mathrm{M}(\mu_k) \le \sum_{k=k_0}^{\f} \int_{\R^d}|x|^2 \D \mu_k(x) \le \sum_{k=k_0}^{\f} \max\set{|a|^2: a\in A_k} < \f.$$
  Note that the series in \eqref{eq-thm-1-3-2} converges in $\R^d$. So, the series
  $$\sum_{k=k_0}^{\f} c(\mu_{k,r}) = \sum_{k=k_0}^{\f} c(\mu_k) = \sum_{k=k_0}^{\f} \frac{1}{\# A_k} \sum_{ a \in A_k} a $$
  also converges in $\R^d$.
  Therefore, the following three series $$ \sum_{k=1}^{\f} \mu_k \big( \R^d \setminus B(r) \big),\quad \sum_{k=1}^{\f} c(\mu_{k,r}), \quad \sum_{k=1}^{\f} \mathrm{M}(\mu_{k,r}) $$ all converge.
  By Theorem \ref{three-series-theorem}, the infinite convolution $\nu$ exists.
\end{proof}

\section{The proof of Theorem \ref{main-theorem}}\label{spectrality}

Let $\set{(N_k,B_k,L_k)}$ be a sequence of Hadamard triples in $\R$. For each integer $n \ge 1$, we write
\begin{equation*}
  \mu_n= \delta_{N_{1}^{-1} B_1} * \delta_{(N_1 N_2)^{-1} B_2} * \cdots * \delta_{(N_1 N_2 \cdots N_n)^{-1} B_n }.
\end{equation*}
\emph{We assume that the sequence $\set{\mu_n}$ converges weakly to a Borel probability measure $\mu$.}
That is, $\mu$ is an infinite convolution, written as
\begin{equation}\label{mu-infinite-convolution}
  \mu = \delta_{N_{1}^{-1} B_1} * \delta_{(N_1 N_2)^{-1} B_2} * \cdots * \delta_{(N_1 N_2 \cdots N_n)^{-1} B_n } * \cdots.
\end{equation}
We rewrite $\mu = \mu_n * \mu_{>n}$,
where $$\mu_{>n} = \delta_{(N_1 N_2 \cdots N_{n+1})^{-1} B_{n+1} } * \delta_{(N_1 N_2 \cdots N_{n+2})^{-1} B_{n+2} } * \cdots$$ is the tail of infinite convolution,
and then we define
\begin{equation}\label{nu-large-than-n}
  \nu_{>n}(\;\cdot\;) = \mu_{>n}\left( \frac{1}{N_1 N_2 \cdots N_n} \; \cdot\; \right).
\end{equation}
In fact, we have that $$\nu_{>n}=\delta_{N_{n+1}^{-1} B_{n+1}} * \delta_{(N_{n+1} N_{n+2})^{-1} B_{n+2}} * \cdots.$$

\begin{definition}\label{def-equipositive}
  We call $\Phi \sse \mcal{P}(\R)$ an equi-positive family if there exists $\ep>0$ and $\delta>0$ such that for $x\in [0,1)$ and $\mu\in \Phi$ there exists an integer $k_{x,\mu} \in \Z$ such that
  $$ |\wh{\mu}(x+y+k_{x,\mu})| \ge \ep,$$
  for all $ |y| <\delta$, where $k_{x,\mu} =0$ for $x=0$.
\end{definition}

The equi-positivity condition was used to study the spectrality of fractal measures with compact support in \cite{An-Fu-Lai-2019,Dutkay-Haussermann-Lai-2019}.
This definition was then generalized in~\cite{LMW} to the current version which is also applicable to study the spectrality of infinite convolutions without compact support.
The following theorem was proved in~\cite{LMW}, and it is the key tool to prove Theorem \ref{main-theorem}.

\begin{theorem}\label{general-result}
  Let $\set{(N_k, B_k , L_k)}$ be a sequence of Hadamard triples in $\R$.
  Suppose that $\mu$ is the infinite convolution defined by \eqref{mu-infinite-convolution}.
  Let $\nu_{>n}$ be defined by \eqref{nu-large-than-n}.
  If there exists a subsequence $\{ n_j \}$ of positive integers such that $\{ \nu_{>n_j} \}$ is an equi-positive family, then $\mu$ is a spectral measure.
\end{theorem}

\begin{proof}[Proof of Theorem \ref{main-theorem}]
  For every $k \ge 1$, let $$L_k = \set{ 0, \frac{N_k}{b_k}, \frac{2 N_k}{b_k}, \cdots, \frac{(b_k-1)N_k)}{b_k} }. $$
  Since $B_k \equiv \set{0,1,\cdots, b_k-1} \pmod{N_k}$, it is obvious that $(N_k,B_k,L_k)$ is a Hadamard triple for each $k \ge 1$.

  For $k \ge 1$, we write
  $$c_k = \#\big(B_k \sm \set{0,1,2,\cdots, b_k-1} \big).$$
  Since $\set{B_k}$ is a sequence of nearly consecutive digit sets with respect to $\set{b_k}$ and $\set{N_k}$, we have
  \begin{equation}\label{sum_fnt}
  \sum_{k=1}^{\f} \frac{c_k}{b_k}<\f.
  \end{equation}
  We divide the set $B_k$ into two parts.
  For every $k \ge 1$, let $B_{k,1} = B_k \cap \set{ 0,1,2, \cdots, b_k -1}$ and $B_{k,2} = B_k \sm \set{0,1,2,\cdots, b_k-1}$. We write $A_k = (N_1 N_2 \cdots N_k)^{-1} B_k$ and
  $$A_{k,1} = (N_1 N_2 \cdots N_k)^{-1} B_{k,1}, \quad A_{k,2} = (N_1 N_2 \cdots N_k)^{-1} B_{k,2}.$$

  First, we prove the existence of the infinite convolution of $\{\delta_{A_k}\}$.
  By Theorem \ref{weak-convergence-infinite-convolution}, it is sufficient to show that
  \begin{equation}\label{sum-finite}
    \sum_{k=1}^{\f} \frac{1}{\# A_k} \sum_{a\in A_{k}} \frac{a}{1+a} < \infty .
  \end{equation}
  Since $A_k=A_{k,1}\cup A_{k,2},$ we divide the summation in \eqref{sum-finite} into two parts and estimate them separately.
  For the first part, we have that
  \begin{eqnarray*}
  \sum_{k=1}^{\f} \frac{1}{\# A_k} \sum_{a\in A_{k,1}} \frac{a}{1+a} &\le& \sum_{k=1}^{\f} \frac{1}{b_k} \sum_{a\in A_{k,1}} a   \\
  &\le& \sum_{k=1}^{\f} \frac{b_k -1}{N_1 N_2 \cdots N_k} \\
  &\le& \sum_{k=1}^{\f} \frac{1}{2^{k-1}} \\
  &<& \infty.
  \end{eqnarray*}
For the second part, by \eqref{sum_fnt}, we have that
\begin{eqnarray*}
\sum_{k=1}^{\f} \frac{1}{\# A_k} \sum_{a\in A_{k,2}} \frac{a}{1+a} &\le& \sum_{k=1}^{\f} \frac{\# A_{k,2}}{\# A_k}  \\
 &\le& \sum_{k=1}^{\f} \frac{c_k}{b_k}\\
  &<&\f.
\end{eqnarray*}
  It follows that \eqref{sum-finite} holds. Therefore, the infinite convolution $\mu$ exists.

  Next, we show that $\mu$ is a spectral measure.
  For a finite subset $B \sse \Z$, we write $$m_B(\xi) = \frac{1}{\# B} \sum_{b \in B} e^{-2\pi i b\xi }.$$
  It is clear that $m_B(\xi)$ is the Fourier transform of the discrete measure $\delta_B$.
  For each $k \ge 1$, we have that for $\xi\not\in \Z$,
  \begin{eqnarray*}
    \left| m_{B_k}(\xi) \right| &\ge&  \left| \frac{1}{b_k} \sum_{j=0}^{b_k-1} e^{-2\pi i j \xi} \right| - \left| \frac{1}{b_k} \sum_{j=0}^{b_k-1} e^{-2\pi i j \xi} - \frac{1}{b_k} \sum_{b \in B_k} e^{-2\pi i b \xi} \right| \\
    & \ge& \left| \frac{1}{b_k} \sum_{j=0}^{b_k-1} e^{-2\pi i j \xi} \right| - \frac{2 \# B_{k,2}}{b_k} \\
    &=&  \frac{1}{b_k} \left| \frac{\sin(b_k \pi \xi)}{\sin(\pi \xi)} \right| - \frac{2c_k}{b_k} .
  \end{eqnarray*}
  By applying inequalities $|\sin x|\leq |x|$ and $\left| \frac{\sin x}{x} \right| \ge 1 - \frac{x^2}{6}$, we obtain that $$\frac{1}{b_k} \left| \frac{\sin(b_k \pi \xi)}{\sin(\pi \xi)} \right| \ge \left| \frac{\sin(b_k \pi \xi)}{ b_k \pi \xi } \right| \ge  1 - \frac{(b_k \pi \xi)^2}{6}. $$
  It follows that for $\xi \not\in \Z$, $$ \left| m_{B_k}(\xi) \right| \ge 1 - \frac{(b_k \pi \xi)^2}{6} - \frac{2c_k}{b_k}. $$
  Note that $m_{B_k}(\xi)=1$ for $\xi \in \Z$.
  Thus, for every $k \ge 1$, we have
  \begin{equation}\label{lower-bound}
    \left| m_{B_k}(\xi) \right| \ge 1 - \frac{(b_k \pi \xi)^2}{6} - \frac{2c_k}{b_k}
  \end{equation}
  for all $\xi \in \R$.

Since the series in \eqref{sum_fnt} is convergent, we may find $n_0 \ge 1$ such that for each $k > n_0$, $$\frac{2c_k}{b_k} < \frac{7}{9}- \frac{2\pi^2}{27}.$$
Note that $\nu_{>n} = \delta_{N_{n+1}^{-1} B_{n+1}} * \delta_{(N_{n+1} N_{n+2})^{-1} B_{n+2} } * \cdots$.
For $n \ge n_0$ and $\xi \in [-2/3,2/3]$, by \eqref{lower-bound}, we have that
  \begin{align*}
    \left| \wh{\nu}_{>n}(\xi) \right| & = \left| \prod_{k=1}^{\f} m_{B_{n+k}}\left( \frac{\xi}{N_{n+1} N_{n+2} \cdots N_{n+k} } \right) \right| \\
    & \ge \prod_{k=1}^{\f} \left[ 1- \frac{1}{6} \left( \frac{ b_{n+k} \pi \xi}{N_{n+1} N_{n+2} \cdots N_{n+k}}  \right)^2  - \frac{2c_{n+k}}{b_{n+k}}\right] \\
    & \ge \prod_{k=1}^\f \left( 1- \frac{2\pi^2}{27} \frac{1}{ 4^{k-1} }  - \frac{2 c_{n+k}}{b_{n+k}}\right).
  \end{align*}
Since for every $k \ge 1$, $$0< \frac{2\pi^2}{27} \frac{1}{ 4^{k-1} }  + \frac{2 c_{n+k}}{b_{n+k}} < \frac{7}{9},$$
by the inequality $1- x \ge e^{-3x}$ for $0< x < 7/9$,  we conclude that
  \begin{eqnarray*}
  \left| \wh{\nu}_{>n}(\xi) \right| &\ge& \exp\left( - \frac{2\pi^2}{9}\sum_{k=1}^{\f} \frac{1}{4^{k-1}} - 6 \sum_{k=1}^{\f} \frac{c_{n+k}}{b_{n+k}} \right)  \\
  &\ge& \exp\left( - \frac{8\pi^2}{27} - 6 \sum_{k=1}^{\f} \frac{c_{k}}{b_{k}} \right)
 \end{eqnarray*}
  for $n \ge n_0$ and $\xi \in [-2/3,2/3]$.

  Let $$ \ep = \exp\left( - \frac{8\pi^2}{27} - 6 \sum_{k=1}^{\f} \frac{c_{k}}{b_{k}} \right), $$ and $\delta=1/6$.
  We define $k_x = 0$ for $x\in [0,1/2]$ and $k_x = -1$ for $x\in [1/2,1)$.
  Therefore, for $x \in [0,1)$ and $n \ge n_0$, we have that $$\left| \wh{\nu}_{>n}(x + y + k_x) \right| \ge \ep$$ for all $|y|< \delta$, since $|x+y+k_x| < 2/3$.
  This implies that the family $\set{\nu_{>n}}_{n=n_0}^\f$ is equi-positive.
  By Theorem \ref{general-result}, the infinite convolution $\mu$ is a spectral measure.
\end{proof}

\section{Intermediate-value property of dimensions}\label{intermediate-value}

In this section, we always assume that $\{b_k\}$ and $\{N_k\}$ are two sequences of positive integers satisfying
\begin{equation*}
  N_k \ge b_k\ge 2, \quad\; b_k \mid N_k,
\end{equation*}
 and $$\sum_{k=1}^{\f} \frac{1}{b_k} < \f.$$
For each $k \ge 1$,  we write
$$B_k = \set{ 0,1, \cdots, b_k -2, b_k -1 + (N_1N_2 \cdots N_k) \cdot(k!)},$$
where $k! = k \cdot (k-1) \cdots 2 \cdot 1$.

Clearly, $\{B_k\}$ is a sequence of nearly consecutive digit sets with respect to $\set{b_k}$ and $\set{N_k}$.
Therefore, by Theorem \ref{main-theorem}, the infinite convolution
\begin{equation}\label{def-mu}
\mu  = \delta_{N_1^{-1} B_1} * \delta_{(N_1N_2)^{-1} B_2} * \cdots * \delta_{(N_1N_2\cdots N_k)^{-1} B_k} * \cdots
\end{equation}
is a spectral measure.

For every $k \ge 1$, we write $A_k = (N_1 N_2 \cdots N_{k})^{-1} B_k$, i.e.,
$$A_k= \set{ 0, \frac{1}{N_1 N_2 \cdots N_{k}}, \cdots, \frac{b_k-2}{N_1 N_2 \cdots N_{k}}, k! + \frac{b_k-1}{N_1 N_2 \cdots N_{k}} }.$$
For a finite subset $S \sse \N$, we define
$$K_S = \set{ \sum_{k \in S} \left( k! + \frac{b_k-1}{N_1 N_2 \cdots N_{k}} \right) + \sum_{k \not\in S}\frac{\ep_k}{N_1 N_2 \cdots N_{k}}: \ep_k \in \set{0,1,\cdots,b_k -2} \text{ for } k \not\in S }.$$
It is easy to verify that $K_S$ is compact
and $$K_S \sse \left[ \sum_{k \in S} k!, 1+ \sum_{k \in S} k! \right).$$

\begin{proposition}\label{prop-support}
Let $\mu$ be given by~\eqref{def-mu}. Then the support of $\mu$ is the disjoint union of all sets $K_S$ where $S\sse \N$ is a finite subset, i.e.,
  \begin{equation}\label{support-of-mu}
    \mathrm{spt}(\mu) = \bigcup_{S \sse \N \text{ finite }} K_S.
  \end{equation}
\end{proposition}
\begin{proof}
Let $$K = \bigcup_{S \sse \N \text{ finite }} K_S. $$
First, we show that the union is disjoint.
By induction, it is easy to check the following inequality $$\sum_{k=1}^{n} k! < (n+1)!$$ for each $n \ge 1$.
Let  $S,T \sse \N$ be two finite subsets such that  $S \ne T$.  It is obvious that
$$\sum_{k\in S} k! \ne \sum_{k \in T} k!.$$
Immediately,   we have that  $K_S \cap K_T = \emptyset$. This implies that the union in~\eqref{support-of-mu} is disjoint.

Next, we prove that the set $K$ is closed.
  Arbitrarily choose a sequence $\{x_j \}_{j=1}^\f \sse K$ such that $x_j$ converges to  $x$ as $j$ tends to $ \f$.
  Since $\{x_j \}_{j=1}^\f$ is bounded, there exists $n\ge 1$ such that $\{x_j \}_{j=1}^\f \sse [0, n!]$.
  Thus, we have that $$\{x_j \}_{j=1}^\f \sse \bigcup_{S \sse \{1,2,\cdots,n\}} K_S.$$
  Since the set $K_S$ is compact for every finite subset $S \sse \N$,  we have that $$x \in \bigcup_{S \sse \{1,2,\cdots,n\}} K_S \sse K.$$
  This implies that $K$ is closed.

  Finally we verify that $\mathrm{spt}(\mu) = K$.
  For $n \ge 1$, we write $\mu_n=\delta_{A_1} * \delta_{A_2}* \cdots *\delta_{A_n}$.
  For every Borel subset $B \sse \R$, we have that
  \begin{equation}\label{measure-mu-n}
    \mu_n(B) = \frac{\#\set{ (a_1,a_2,\cdots,a_n) \in A_1 \times A_2 \times \cdots \times A_n: a_1 + a_2 + \cdots + a_n \in B }}{b_1 b_2 \cdots b_n}.
  \end{equation}
  Since $ a_1 + a_2 + \cdots + a_n \in K$ for every $n$-tuple $(a_1, a_2,\cdots,a_n) \in A_1 \times A_2 \times \cdots \times A_n$, it is obvious that  $\mu_n(K)=1$ for each $n \ge 1$.
  Since $\{\mu_n\}$ converges weakly to $\mu$ and $K$ is closed, by Theorem \ref{weak-convergence-theorem}   (ii), we have that $\mu(K)=1$. It follows that $\mathrm{spt}(\mu) \sse K$.

  It remains to show that $K\sse \mathrm{spt}(\mu) $. Fix a finite subset $S\sse \N$ and $x\in K_S$.
  We write $$x = \sum_{k \in S} k! + \sum_{k=1}^{\f} \frac{\eta_k}{N_1 N_2 \cdots N_k},$$
  where $\eta_k \in \set{0,1,\cdots, b_k -1}$ for each $k \ge 1$.
  Let $k_0 = \max\set{k: k\in S}$. For each $\delta >0$, there exists $\ell_0 > k_0$ such that
  $$\sum_{k=\ell_0 + 1}^{\f} \frac{b_k-1}{N_1 N_2 \cdots N_k} < \delta.$$
  We write $$\wt{x} = \sum_{k \in S} k! + \sum_{k=1}^{\ell_0} \frac{\eta_k}{N_1 N_2 \cdots N_k}
  $$
  and $$K_{x,\delta} = \wt{x} + \set{ \sum_{k=\ell_0 + 1}^{\f} \frac{\ep_k}{N_1 N_2 \cdots N_k}: \ep_k \in \set{0,1,\cdots,b_k-2} \text{ for } k \ge \ell_0+1 }. $$
  It is clear that $K_{x,\delta}$ is compact, and $K_{x,\delta} \sse (x-\delta, x+\delta)$.
  For $n > \ell_0$, by \eqref{measure-mu-n}, we have that
  $$\mu_{n}(K_{x,\delta}) = \frac{1}{b_1 b_2 \cdots b_{\ell_0}} \prod_{k=\ell_0 +1}^{n} \left( 1- \frac{1}{b_k} \right).$$
  By Theorem \ref{weak-convergence-theorem} (ii), we obtain that
  $$\mu(K_{x,\delta}) \ge \limsup_{n \to \f} \mu_n(K_{x,\delta}) = \frac{1}{b_1 b_2 \cdots b_{\ell_0}} \prod_{k=\ell_0 +1}^{\f} \left( 1- \frac{1}{b_k} \right) >0 ,$$
  and it follows that $$\mu\big( (x-\delta,x+\delta) \big) \ge \mu(K_{x,\delta})  >0.$$
  Since $\delta$ is arbitrary, we have that $x\in \mathrm{spt}(\mu)$.
  It follows that $K_S \sse \mathrm{spt}(\mu)$, and thus, $K \sse \mathrm{spt}(\mu)$.
  Therefore, we conclude that $\mathrm{spt}(\mu) = K$.
\end{proof}

The following theorem is the well-known Stolz-Ces\`{a}ro theorem which will be used to calculate dimensions.
\begin{theorem}\label{stolz-cesaro-theorm}
  Let $\set{\beta_k}_{k=1}^\f$ be a sequence of positive real numbers satisfying $\sum_{k=1}^{\f} \beta_k = \f$, and let $\set{\alpha_k}_{k=1}^\f$ be any given sequence of real numbers.
  Then we have that $$\liminf_{n \to \f} \frac{\alpha_1 + \alpha_2 + \cdots + \alpha_n}{\beta_1 + \beta_2 + \cdots + \beta_n} \ge \liminf_{k \to \f} \frac{\alpha_k}{\beta_k}$$
  and $$\limsup_{n \to \f} \frac{\alpha_1 + \alpha_2 + \cdots + \alpha_n}{\beta_1 + \beta_2 + \cdots + \beta_n} \le \limsup_{k \to \f} \frac{\alpha_k}{\beta_k}.$$
  In particular, if the limit of the sequence $\set{\alpha_k / \beta_k}_{k=1}^\f$ exists, then
  $$ \lim_{n\to \f} \frac{\alpha_1 + \alpha_2 + \cdots + \alpha_n}{\beta_1 + \beta_2 + \cdots + \beta_n} = \lim_{k \to \f} \frac{\alpha_k}{\beta_k}. $$
\end{theorem}
Next, we state the dimension formula  for the support of $\mu$.

\begin{proposition}\label{dimension-proposition}
Let $\mu$ be given by~\eqref{def-mu}. Then we have that
  $$\dim_H \mathrm{spt}(\mu) = \liminf_{k \to \f} \frac{\log b_1 b_2 \cdots b_k}{\log N_1 N_2 \cdots N_k N_{k+1} - \log b_{k+1}},$$
and
$$\dim_P \mathrm{spt}(\mu) = \limsup_{k \to \f} \frac{\log b_1 b_2 \cdots b_k}{\log N_1 N_2 \cdots N_k }.$$
\end{proposition}
\begin{proof}
We only prove the Hausdorff dimension formula, and the proof of packing dimension is almost identical.
We write $$s = \liminf_{k \to \f} \frac{\log b_1 b_2 \cdots b_k}{\log N_1 N_2 \cdots N_k N_{k+1} - \log b_{k+1}}.$$

Fix a finite subset $S\sse \N$.  Let $k_0 = \max\set{k: k \in S}$.
We define
  \begin{equation}\label{the-set-C}
    C =\set{ \sum_{k=1}^{\f} \frac{\ep_k}{N_{k_0+1} N_{k_0+2} \cdots N_{k_0 +k}}: \ep_k \in \set{0,1,\cdots, b_{k_0+k} -2 } \text{ for } k \ge 1 }.
  \end{equation}
Since $C$ is the partial homogeneous Cantor set, by the dimension formula for partial homogeneous Cantor sets in \cite{Feng-Wen-Wu-1997},
we have that
  $$\dim_H C = \liminf_{k \to \f} \frac{\log (b_{k_0+1}-1) (b_{k_0+2}-1) \cdots (b_{k_0+k}-1)}{\log N_{k_0+1} N_{k_0+2} \cdots N_{k_0+k} N_{k_0+k+1} - \log (b_{k_0+k+1}-1)}.$$
  Since $b_k \to \f$ as $k \to \f$, by Theorem \ref{stolz-cesaro-theorm}, we have that
  $$\lim_{k \to \f} \frac{ \log (b_{k_0+1}-1) (b_{k_0+2}-1) \cdots (b_{k_0+k}-1) }{ \log b_1 b_2 \cdots b_{k_0 +k} } = \lim_{k\to \f} \frac{\log( b_{k_0+k}-1 )}{\log b_{k_0+k}} =1.$$
  Together with
  $$\lim_{k \to \f} \frac{\log N_1 N_2 \cdots N_{k_0+k} N_{k_0+k+1} - \log b_{k_0+k+1} }{\log N_{k_0+1} N_{k_0+2} \cdots N_{k_0+k} N_{k_0+k+1} - \log (b_{k_0+k+1}-1)}=1, $$
  we conclude that
  $$ \dim_H C = \liminf_{k \to \f} \frac{ \log b_1 b_2 \cdots b_{k_0+k} }{\log N_1 N_2 \cdots N_{k_0+k} N_{k_0 + k +1} - \log b_{k_0 + k +1}} = s. $$
  Since $K_S$ is the union of finitely many of similar copies of $C$,  we have that $\dim_H K_S = \dim_H C =s$.
  Finally, by Proposition~\ref{prop-support} and the countable stability of Hausdorff dimension, we obtain that
  $\dim_H \mathrm{spt}(\mu) = s$.
\end{proof}

\begin{proof}[Proof of Theorem \ref{dimension-theorem}]
  We prove this theorem by choosing two appropriate sequences $\set{b_k}$ and $\set{N_k}$.
  Let $b_1=2$ and $b_k = k^2$ for $k \ge 2$.
  Choose a strictly increasing sequence of integers $\{ \ell_j \}_{j=1}^\f$ such that $\ell_1 =0$ and
  \begin{equation}\label{eq-proof-1-7-2}
    \lim_{j \to \f} \frac{\ell_j \log \ell_j}{\ell_{j+1}-\ell_j} =0.
  \end{equation}
  The family of functions $g_\gamma: \N \to \N$ for $0\le \gamma \le 1$ is defined by
  \begin{equation*}
    g_\gamma(n) =
      \begin{cases}
        n^{1+\lfloor \log n \rfloor}, & \mbox{if } \gamma=0; \\
        \lfloor n^{1/\gamma -1} \rfloor n, & \mbox{if } 0< \gamma < 1; \\
        2n, & \mbox{if } \gamma=1;
      \end{cases}
  \end{equation*}
  where $\lfloor x \rfloor$ denotes the integral part of a real number $x$.
  Fix $0\le \alpha \le \beta\le 1$.
  For $k \ge 1$, let
  \begin{equation*}
    N_k =
    \begin{cases}
      g_\alpha(b_k), & \mbox{if } \ell_j < k \le \ell_{j+1} \text{ for some odd $j \in \N$}; \\
      g_\beta(b_k), & \mbox{if } \ell_j < k \le \ell_{j+1} \text{ for some even $j \in \N$}.
    \end{cases}
  \end{equation*}

  Clearly, for each $0\le \gamma \le 1$,  we have that (i) $n \mid g_\gamma(n)$ for all $n\ge 1$; (ii)
  \begin{equation}\label{eq-proof-1-7-1}
    \lim_{n \to \f} \frac{\log n}{\log g_\gamma(n)} = \gamma;
  \end{equation}
  and (iii) there exists $n_0 = n_0(\gamma) \ge 1$ such that $g_\gamma(n) \le n^{1+\log n}$ for all $n \ge n_0$.
  Thus, we have that $b_k \mid N_k$ for all $k \ge 1$.
  It is obvious that $$\sum_{k=1}^{\f} \frac{1}{b_k} < \f. $$
  Therefore, the infinite convolution $\mu$ defined by \eqref{def-mu} is a spectral measure.
  By Proposition \ref{prop-support}, the support of $\mu$ is \emph{not} compact.

  Next we calculate the dimensions of the support of $\mu$.
  By Theorem \ref{stolz-cesaro-theorm} and \eqref{eq-proof-1-7-1}, we have that
  $$ \liminf_{k \to \f} \frac{\log b_1 b_2 \cdots b_k}{\log N_1 N_2 \cdots N_k }\ge \liminf_{k\to \f} \frac{\log b_k}{\log N_k} = \alpha, $$
  and $$ \limsup_{k \to \f} \frac{\log b_1 b_2 \cdots b_k}{\log N_1 N_2 \cdots N_k }\le \limsup_{k\to \f} \frac{\log b_k}{\log N_k} = \beta. $$

  For sufficiently large $j$, we have that
  \begin{align*}
    \frac{\log b_1 b_2 \cdots b_{\ell_{2j+1}} }{\log N_1 N_2 \cdots N_{\ell_{2j+1}} }
    & \ge \frac{  \log b_{\ell_{2j}+1} b_{\ell_{2j}+2} \cdots b_{\ell_{2j+1}} }{ \log N_1 N_2 \cdots N_{\ell_{2j}} + \log N_{\ell_{2j}+1} N_{\ell_{2j}+2} \cdots N_{\ell_{2j+1}} } \\
    & \ge \frac{  \log b_{\ell_{2j}+1} b_{\ell_{2j}+2} \cdots b_{\ell_{2j+1}} }{ \ell_{2j} \log g_\alpha(b_{\ell_{2j}}) + \log g_\beta(b_{\ell_{2j}+1}) g_\beta(b_{\ell_{2j}+2}) \cdots g_\beta(b_{\ell_{2j+1}}) }.
  \end{align*}
  Together with \begin{align*}
    \frac{ \ell_{2j} \log g_\alpha(b_{\ell_{2j}}) }{ \log g_\beta(b_{\ell_{2j}+1}) g_\beta(b_{\ell_{2j}+2}) \cdots g_\beta(b_{\ell_{2j+1}}) } & \le \frac{ \ell_{2j} ( 1 +\log b_{\ell_{2j}} ) \log b_{\ell_{2j}} }{ (\ell_{2j+1} - \ell_{2j}) \log g_\beta(b_{\ell_{2j}+1}) } \\
    &\le \frac{\ell_{2j} ( 1 + 2 \log \ell_{2j})}{ \ell_{2j+1} - \ell_{2j}} \longrightarrow 0,
  \end{align*}
  which follows from \eqref{eq-proof-1-7-2},
  and $$\lim_{j \to \f} \frac{ \log b_{\ell_{2j}+1} b_{\ell_{2j}+2} \cdots b_{\ell_{2j+1}} }{ \log g_\beta(b_{\ell_{2j}+1}) g_\beta(b_{\ell_{2j}+2}) \cdots g_\beta(b_{\ell_{2j+1}}) } = \beta,$$
  which may be deduced from \eqref{eq-proof-1-7-1},
  we obtain that
  $$ \limsup_{k \to \f} \frac{\log b_1 b_2 \cdots b_k}{\log N_1 N_2 \cdots N_k } \ge \limsup_{j \to \f} \frac{\log b_1 b_2 \cdots b_{\ell_{2j+1}} }{\log N_1 N_2 \cdots N_{\ell_{2j+1}} } \ge \beta.$$
  Thus, by Proposition \ref{dimension-proposition}, we have that $$ \dim_P \mathrm{spt}(\mu)= \limsup_{k \to \f} \frac{\log b_1 b_2 \cdots b_k}{\log N_1 N_2 \cdots N_k }= \beta. $$

  On the other hand, we have that
  \begin{align*}
    \frac{\log b_1 b_2 \cdots b_{\ell_{2j}} }{\log N_1 N_2 \cdots N_{\ell_{2j}} }
    & \le \frac{ \log b_1 b_2 \cdots b_{\ell_{2j-1}} + \log b_{\ell_{2j-1}+1} b_{\ell_{2j-1}+2} \cdots b_{\ell_{2j}} }{ \log N_{\ell_{2j-1}+1} N_{\ell_{2j-1}+2} \cdots N_{\ell_{2j}} } \\
    & \le \frac{ \ell_{2j-1} \log b_{\ell_{2j-1}} }{ (\ell_{2j} - \ell_{2j-1}) \log g_\alpha(b_{\ell_{2j-1}+1}) } +
       \frac{ \log b_{\ell_{2j-1}+1} b_{\ell_{2j-1}+2} \cdots b_{\ell_{2j}} }{\log g_\alpha(b_{\ell_{2j-1}+1}) g_\alpha(b_{\ell_{2j-1}+2}) \cdots g_\alpha(b_{\ell_{2j}}) } \\
    & \le \frac{ \ell_{2j-1} }{ \ell_{2j} - \ell_{2j-1} } +
       \frac{ \log b_{\ell_{2j-1}+1} b_{\ell_{2j-1}+2} \cdots b_{\ell_{2j}} }{\log g_\alpha(b_{\ell_{2j-1}+1}) g_\alpha(b_{\ell_{2j-1}+2}) \cdots g_\alpha(b_{\ell_{2j}}) }.
  \end{align*}
  Thus, by \eqref{eq-proof-1-7-2} and \eqref{eq-proof-1-7-1}, we obtain that
  $$ \liminf_{k \to \f} \frac{\log b_1 b_2 \cdots b_k}{\log N_1 N_2 \cdots N_k } \le \liminf_{j \to \f} \frac{\log b_1 b_2 \cdots b_{\ell_{2j}} }{\log N_1 N_2 \cdots N_{\ell_{2j}} }\le \alpha.$$
  It follows that $$ \liminf_{k \to \f} \frac{\log b_1 b_2 \cdots b_k}{\log N_1 N_2 \cdots N_k }= \alpha. $$
  Note that $N_{k+1} \le b_{k+1}^{1+\log b_{k+1}}$ for all large enough $k$. So,
  $$ \frac{\log N_{k+1} - \log b_{k+1} }{ \log N_1 N_2 \cdots N_k }  \le \frac{(\log b_{k+1})^2}{\log b_1 b_2 \cdots b_k} \le \frac{2 \big(\log(k+1)\big)^2}{\log (k!)} \longrightarrow 0. $$
  It follows from Proposition \ref{dimension-proposition} that
  $$\dim_H \mathrm{spt}(\mu) = \liminf_{k \to \f} \frac{\log b_1 b_2 \cdots b_k}{\log N_1 N_2 \cdots N_k N_{k+1} - \log b_{k+1}} = \alpha.$$

  It remains to show that the measure $\mu$ and the Lebesgue measure is mutually singular.
  If $0\le \alpha < 1$, since $\dim_H \mathrm{spt}(\mu) =\alpha < 1$, it is obvious that $\mu$ is singular.
  If $\alpha = \beta =1$, then we have that $N_k = 2 b_k$ for all $k \ge 1$, and the set $C$ defined in \eqref{the-set-C} is a Lebesgue null set. Thus, the Lebesgue measure of $\mathrm{spt}(\mu)$ is zero, that is, $\mu$ is singular.
  The proof is completed.
\end{proof}

\section*{Acknowledgements}

The authors are grateful to the referees for their helpful comments.
Wenxia Li is supported by NSFC No. 12071148, 11971079 and Science and Technology Commission of Shanghai Municipality (STCSM)  No.~18dz2271000.
Jun Jie Miao is supported by Science and Technology Commission of Shanghai Municipality (STCSM)  No.~18dz2271000.

\end{document}